\documentclass{amsart}

\usepackage{amsmath,amsfonts,amsthm,amssymb}
\usepackage[alphabetic]{amsrefs}
\usepackage[OT4]{fontenc}

\usepackage[utf8]{inputenc}
\usepackage[polish,english]{babel}
\usepackage{graphicx}
\DeclareGraphicsExtensions{eps, EPS}
\usepackage{epstopdf}

\newtheorem{theorem}{Theorem}[section]

\newtheorem{lemma}[theorem]{Lemma}
\newtheorem{lem}[theorem]{Lemma}
\newtheorem{theo}[theorem]{Theorem}

\newtheorem{prop}[theorem]{Proposition}

\newtheorem{cor}[theorem]{Corollary}

\theoremstyle{definition}

\newtheorem{rem}[theorem]{Remark}

\newtheorem{defi}[theorem]{Definition}

\newcommand{\cay}{\widetilde{X}}

\theoremstyle{definition}

\theoremstyle{remark}

\numberwithin{equation}{section}

\begin{document}

\title{The square model for random groups}

\author{Tomasz Odrzyg{\'o}{\'z}d{\'z}}
\address{Institute of Mathematics, Polish Academy of Science,
Warsaw, {\'S}niadeckich 8}
\email{tomaszo@impan.pl}

\begin{abstract}
We introduce a new random group model called the \textit{square model}: we quotient a free group on $n$ generators by a random set of relations, each of which is a reduced word of length four. We prove, as in the Gromov model introduced in \cite{gro93}, that for densities $> \frac{1}{2}$ a random group in the square model is trivial with overwhelming probability and for densities $<\frac{1}{2}$ a random group is with overwhelming probability hyperbolic. Moreover we show that for densities $\frac{1}{4} < d < \frac{1}{3}$ a random group in the square model does not have Property (T). Inspired by the results for the triangular model we prove that for densities $<\frac{1}{4}$ in the square model, a random group is free with overwhelming probability. We also introduce abstract diagrams with fixed edges and prove a generalization of the isoperimetric inequality.
\end{abstract}

\maketitle


\section{Introduction}

Each group can be obtained by quotienting a free group by a normal subgroup generated by a set of relators. In \cite{gro93} Gromov introduced  the notion of a random finitely presented group on $m \geq 2$ generators at density $d \in (0,1)$. The idea was to fix a set of $m$ generators and consider presentations with $(2m-1)^{dl}$ relators, each of which is a random reduced word of length $l$. Gromov investigated the properties of random groups when $l$ goes to infinity. We say that a property occurs in the Gromov density model with \textit{overwhelming probability} if the probability that a random group has this property converges to $1$ when $l \rightarrow \infty$. Significant results of this theory are the following: for densities $> \frac{1}{2}$ a random group is trivial with overwhelming probability \cite[Theorem 11]{oll05}; for densities $< \frac{1}{2}$ a random group is, with overwhelming probability, infinite, hyperbolic and torsion-free \cite[Theorem 11]{oll05}; for densities $<\frac{1}{5}$ a random group does not have Property (T) with overwhelming probability \cite[Corollary 7.5]{ow11}.

One modification of Gromov's idea is the triangular model: length of relators in the presentation is always 3, but we let the number of generators go to infinity. Precisely, for a fixed density $d$, we consider a presentation on $n$ generators with $n^{3d}$ relations, each of which is a random reduced word of length $3$. We say that some property occurs in the triangular model with \textit{overwhelming probability} if the probability that a random group has this property converges to $1$ when $n \rightarrow \infty$. This model was introduced in \cite{zuk96} and further studied in \cite{kot11}. The triangular model was a way to prove that random groups in the Gromov density model for densities $> \frac{1}{3}$ have Property (T) with overwhelming probability \cite[Theorem B]{kot11}, \cite{zuk03}.

We introduce a new random group model by considering a random set of relations, each of which is a random reduced word of length four. The following notation will be used in the whole paper

Consider the set $A_n = \{a_1, \dots, a_n \}$, which we will treat as an alphabet. Let $W_n$ be the set of positive words of length 4 over $A_n$ and $W'_n$ be the set of all cyclically reduced words of length 4 over  $A_n$. Note that $|W_n| = n^4$ and $|W'_n| = (2n-1)^4$ up to a multiplicative constant. By $F_n$ we will denote the free group generated by the elements of $A_n$. By \textit{relators} we will understand words over generators and by \textit{relations} equalities holding in the group.

\begin{defi}\label{def:gcar}
For $d \in (0,1)$ let us choose randomly, with uniform distribution, a subset $R_n \subset W_n$ such that $|R_n| = \lfloor n^{4d} \rfloor$. Quotienting $F_n$ by the normal closure of the set $R_n$, we obtain a \textit{random group in the positive square model at density d}.
\end{defi}

\begin{defi}\label{def:overwhelming}
We say that property $P$ occurs in the positive square model with \emph{overwhelming probability} if the probability that a random group has property $P$ converges to $1$ when $n \rightarrow \infty$.
\end{defi}

The most important group properties which we consider are: being trivial, being free, being hyperbolic and having Property (T). We prove, as in the Gromov model, that for densities $>\frac{1}{2}$ a random group in the positive square model is trivial with overwhelming probability (Theorem \ref{thm:triv}) and that for densities $<\frac{1}{2}$ a random group is with overwhelming probability hyperbolic (Corollary \ref{cor:hyp}). Moreover we show that for densities $<\frac{1}{3}$ a random group in the positive square model does not have Property (T) (Theorem \ref{thm:nT}). Inspired by the results in the triangular model we prove that for densities $<\frac{1}{4}$ in the positive square model a random group is free with overwhelming probability (Theorem \ref{thm:free}). We also introduce abstract diagrams with fixed  edges (Definition \ref{def:fix}) and prove a generalization of the isoperimetric inequality (Theorem \ref{thm:lab}).

It is not known what is the optimal density threshold for a property of not having Property (T) in the Gromov density model. This model seems to be much harder to analyze than the triangular model, where we know that for densities $< \frac{1}{3}$ a random group is free with overwhelming probability \cite[Proposition 30]{oll05} and that for densities $> \frac{1}{3}$ a random group has Property (T) with overwhelming probability \cite[Theorem A]{kot11}, \cite{zuk03}.

In our model we expect to find the critical density threshold for having Property (T) in further research. We also expect that for densities $< \frac{1}{3}$ a random group in the square model can be cubulated (see \cite{ow11} for discussion about cubulating random groups in the Gromov density model). The advantage of the positive square model is that since the length of relators is even, we can consider the hypergraphs in the presentation complex and Cayley complex of the random group (Definition \ref{def:hyp}), which is not possible in the triangular model. As we will see, hypergraphs are a useful tool to investigate many group theoretic and topological properties.

We have decided to consider as relators only positive words to avoid technical annoyances, but we will show that all of our results remain true in a model where we allow non-positive words. Firstly, we will define this model:

\begin{defi}\label{def:gcar}
For $d \in (0,1)$ let us choose randomly, with uniform distribution, a subset $R_n \subset W'_n$ such that $|R_n| = \lfloor (2n-1)^{4d} \rfloor$. Quotienting $F_n$ by the normal closure of the set $R_n$, we obtain a \textit{random group in the square model at density d}.
\end{defi}

\textbf{Acknowledgment}. I would like to thank Piotr Przytycki for suggesting the topic of this thesis, for his help, valuable suggestions and many interesting discussions. I am also grateful to Irena Danilczuk--Jaworska, Jerzy Trzeciak and Micha{\l}  Kotowski for improving this paper and many corrections.


\section{Triviality}

Firstly we are going to investigate a case where there are many relations. Inspired by the results in Gromov's theory we have proved the theorem stating that when density is greater than~$\frac{1}{2}$, random groups in positive square models are as trivial as possible (with overwhelming probability). Let us determine ,,how trivial'' can such a group be.

The most obvious case is when all the generators are equal. In that case we indeed have only one generator, g, and, in the positive square model, only one relation $g^4 = e$ (by $e$ we denote the neutral element). It means that our group is isomorphic to $\mathbb{Z}_4$. Hence $\mathbb{Z}_4$ is the smallest group which can be obtained in the positive square models. We will call it \textit{trivial}.

Our goal in this section is to prove the following theorem:

\begin{theo}\label{thm:triv}
In the positive square model at density $d >\frac{1}{2}$ a random group is trivial (in the sense described above) with overwhelming probability.
\end{theo}

Before providing the proof we need to introduce random graphs and recall several facts about them.

\begin{defi}[Erd{\"o}s--R{\'e}nyi random graph] \label{def:random_graph}
$G(n,m)$ is the graph obtained by sampling uniformly from all graphs with $n$ vertices and $m$ edges.
\end{defi}

\begin{defi}[Gilbert random graph]
$G(n,p)$ is the random graph obtained by starting with vertex set $V = \{1, 2, \dots, n\}$, letting $0 \leq p \leq 1$, and connecting each pair of vertices by an edge with probability $p$.
\end{defi}

In general $m$ and $p$ are functions of $n$. We will say that a random graph in $G(n,p)$ model has some property \textit{asymptotically almost surely} if the probability that this property occurs converges to $1$ when $n \rightarrow \infty$.

\begin{rem}[\cite{er60}]\label{rem:graphs}
There is strong links between the two models. First note that the expected number of edges in $G(n,p)$ model equals $\left(n \atop 2 \right) p$. If $P$ is any graph property which is monotone with respect to the subgraph ordering (meaning that if $A$ is a subgraph of $B$ and $A$ satisfies $P$, then $B$ satisfies $P$ as well), then the statements ``P holds in the Erd{\"o}s and R{\'e}nyi $G(n, p)$ model with overwhelming probability'' and ``P holds in the Gilbert $G(n, \lfloor \left(n \atop 2 \right) p \rfloor)$ model with overwhelming probability'' are equivalent if $pn^2 \rightarrow \infty$.

Graph properties which are monotone in the above sense and are relevant for us are: being connected and having a cycle of odd length.
\end{rem}

Gilbert's model is much easier for calculations than Erd{\"o}s--R{\'e}nyi's one, so we will prove two lemmas in Gilbert's model.

\begin{lemma}\label{lem:conn}
Let $G$ be a random graph in $G(n,p)$ model, where $p \geq n^{\delta - 1}$ for some $\delta > 0$. Then $G$ is connected asymptotically almost surely.
\end{lemma}

More general statements about connectedness of random graphs can be found for example in \cite[Theorem 7.3]{bela}, but since we repeat the argument in the proof of Lemma \ref{lem:cycle} we present here a simple proof of Lemma \ref{lem:conn}.

\begin{proof}
Denote by $V$ the set of vertices of $G$. Note that disconnectedness means that there exist two nonempty sets $S, T \subset V$ such that $S \cup T = V$, $S \cap T = \emptyset$ and there is no edge between $S$ and $T$. For a fixed $S$ and $T$ the probability that there is no edge between $S$ and $T$ equals $(1-p)^{|S|\cdot |T|}$. Hence the probability $P_d$ of disconnectedness can be estimated (as of now we assume $n > 2$):

$$ P_d \leq \sum_{l=1}^n \left(n \atop l \right)(1 - p)^{l(n-l)} $$

The right hand side can be estimated as follows

\begin{equation}\label{eq:num}
\sum_{l=1}^n \left(n \atop l \right)(1 - p)^{l(n-l)} \leq 2\sum_{l=1}^{\lceil \frac{n}{2} \rceil} \left(n \atop l \right)(1 - p)^{l(n-l)}
\leq  2\sum_{l=1}^{\lceil \frac{n}{2} \rceil} n^l (1 - p)^{l \lfloor \frac{n}{2} \rfloor}
\end{equation}

From our assumption on $p$ we know that

$$ n(1 - p)^{\lfloor \frac{n}{2} \rfloor} \leq n (1 - n^{\delta - 1})^{\frac{n}{2} - 1}.$$

Let us denote $z_n = n (1 - n^{\delta - 1})^{\frac{n}{2} - 1}$. Because $z_n > 0$ instead of proving that $\displaystyle \lim_{n \rightarrow \infty} z_n = 0$ we can prove $\displaystyle \lim_{n \rightarrow \infty} \ln z_n = -\infty$. It is well known that $|\ln(1-x)| > x$ for $x \in (0,1)$. Hence, we estimate:

$$\displaystyle \ln z_n = \ln n + \left( \frac{n}{2} - 1 \right) \ln \left( 1 - \frac{n^{\delta}}{n} \right) < \ln n - \left( \frac{n}{2} - 1 \right)\frac{n^{\delta}}{n}.$$

Therefore, $\displaystyle \lim_{n \rightarrow \infty} \ln z_n = -\infty$. Thus the geometric series on the right hand side of (\ref{eq:num}) converges to 0 when $n \rightarrow \infty$.
\end{proof}

\begin{lemma}\label{lem:cycle}
Let $G$ be a random graph in $G(n,p)$ model, where $p \geq n^{\delta - 1}$ for some $\delta > 0$. Then asymptotically almost surely there is a cycle of odd length in $G$.
\end{lemma}

\begin{proof}
We will estimate the probability that our graph is connected and has no cycle of odd length. First we will prove that if the graph is connected and has no cycle of odd length then it is bipartite:

Let $V$ be the set of vertices of $G$. Denote by $T$ the spanning tree of $G$. Such a tree exists from connectedness. Fix some vertex $v \in V$. Then we can define sets:
$$ A = \{ u \in V \emph{ : there is a path in } T  \emph{ between u and v  of odd length} \}$$
$$ B = \{ u \in V \emph{ : there is a path in } T  \emph{ between u and v  of even length} \}.$$

Sets $A$ and $B$ make a partition of $V$:  $A \cap B = \emptyset$, $A \cup B = V$. Hence, it is sufficient to show that the graph is connected and not bipartite.

For fixed sets $A, B \subset V$ the probability that there is no edge with both ends in $A$ equals $(1-p)^{\left(|A| \atop 2 \right)}.$ Thus the probability $P_c$ that there is no cycle of odd length can be estimated:

\begin{equation}\label{cyc}
P_c \leq \sum_{l=1}^{n-1} \left(n \atop l \right) (1-p)^{\left(n \atop l \right) \left(n \atop n - l \right) } \leq \sum_{l=1}^{n-1} n^l (1 - p)^{nl}
\end{equation}

The summation is over all sets $A$ and $B$. In the last inequality we used the fact: $\left(n \atop l \right) \left(n \atop n - l \right) > nl$. In the proof of Lemma \ref{lem:conn} we have already shown that the right hand side of (\ref{cyc}) converges to 0  when $n \rightarrow \infty$. This ends the proof.
\end{proof}

\begin{rem}\label{rem:path}
Let $G$ be a connected graph that has a cycle of odd length. Let $x, y$ be vertices of $G$. Then there exists an edge path in $G$ of even length joining $x$ and $y$.
\end{rem}

\begin{proof}
Denote by $\gamma_c$ the closed edge path in $G$ of odd length. Let $v$ be the beginning vertex of $\gamma_c$. From connectedness of $G$ there exist edge paths $\gamma_{xv}$ from $x$ to $v$ and $\gamma_{vy}$ from $v$ to $y$. If $|\gamma_{xv} \cup \gamma_{vy}|$ is an even number then $\gamma_{xv} \cup \gamma_{vy}$ is the desired edge path joining $x$ and $y$. If $|\gamma_{xv} \cup \gamma_{vy}|$ is an odd number then $\gamma_{xv} \cup \gamma_c \cup \gamma_{vy}$ is the desired edge path.
\end{proof}

Now we are able to provide the proof of our main statement.

\begin{proof}[Proof of Theorem~\ref{thm:triv}]

Let $D_n$ be the set of positive words of length 2 over $A_n$, i.e. $D_n = \{a_i a_j\ |\  i, j = 1, \dots, n \}$. The set of positive words of length 4 over $A_n$ coincides with the set of positive words of length 2 over $D_n$.

Let $R$ be the set of relators in the presentation of the random group and denote by $R_0 \subset R$ the set of elements of $R$ of form $a_i a_j a_i a_j$ for $1 \leq i, j \leq n$. Let $P_k$ be the probability that group $\left< A_n | R - R_0 \right>$ is trivial. We denote by $\widetilde{P}_k$ the probability that $|R_0| = k$. Then from the Bayes formula the probability that a random group is trivial is greater than:

\begin{equation}\label{eq:estyma}
\sum_{k=0}^{n^2} \widetilde{P}_k P_k
\end{equation}

It can be easily seen that $P_0 > P_1 > \dots > P_{n^2}$ and $\sum_{k=0}^{n^2} \widetilde{P}_k = 1$. We will prove that $P_{n^2} \rightarrow 1$ when $n \rightarrow \infty$, which will imply that (\ref{eq:estyma}) converges to $1$ when $n \rightarrow \infty$

Let us assume that $|R_0| = n^2$. Consider the following graph $G$ with the set of vertices $D_n$: when the relator $a_i a_j a_k a_l$ belonging to $R - R_0$ is drawn, we add the edge in $G$ connecting the vertices $a_i a_j$ and $a_k a_l$. Thus $G$ is a random graph in $G(n^2, \lfloor n^{4d} \rfloor - n^2)$ Erd{\"o}s and R{\'e}nyi model.

Let us consider a random graph $G'$ in $G \left(n^2, \frac{1}{\left(n \atop 2 \right)}\left( \lfloor n^{4d} \rfloor - n^2 \right)\right)$ Gilbert model. From Lemmas \ref{lem:conn} and \ref{lem:cycle} we know that asymptotically almost surely graph $G'$ is connected and has a cycle of odd length. Hence, from Remark \ref{rem:graphs} asymptotically almost surely graph $G$ is connected and has a cycle of odd length. Therefore, from Remark \ref{rem:path} asymptotically almost surely for any two vertices of $G$ there is a path of even length joining them.

An edge between the vertices $a_i a_j$ and $a_k a_l$ of $G$ corresponds to the relation $a_i a_j = (a_k a_l)^{-1}$ in our random group. An adjacent edge connecting  $a_k a_l$ and $a_t a_s$ implies that $a_i a_j =a_t a_s$. Therefore, by induction, if there is a path of even length joining $a_i a_j$ and  $a_k a_l$ we have that $a_i a_j = a_k a_l$. According to the previous observations about graph $G$ this means that with overwhelming probability all words $a_i a_j$ are equal. In particular, for any $i, j, k$ we have $a_i a_k = a_j a_k$ which implies that $a_i = a_j$. Therefore, all generators are equal. This ends the proof.
\end{proof}

\subsection{Triviality in the square model}

Our goal is to prove the following

\begin{theo}\label{thm:square_triv}
In the square model at density $d >\frac{1}{2}$ a random group is trivial with overwhelming probability.
\end{theo}

First we will prove

\begin{lem}\label{lem:ddp}
Let $W_n \subset W'_n$ be the set of positive words of length four over $A_n$. Let $G = \left< A_n | R_n \right>$ be the random group in the square model at density $d$. Then, for any $d' < d$:

$$\mathbb{P}(|R_n \cap W_n| > n^{4d'}) \rightarrow 1$$
as $n \rightarrow \infty$.
\end{lem}

\begin{proof}
First we will prove that
\begin{equation}\label{lemma:gromov_lemma}
\mathbb{P} \left( |R_n \cap (W'_n - W_n)| > \frac{16}{17}\lfloor (2n-1)^{4d} \rfloor \right) \rightarrow 0,
\end{equation}
when $n \rightarrow \infty$.

Drawing at random a set of relators can be treated as sampling without replacement: we draw the first relator $r_1$, then from the set $W'_n - r_1$ we draw the next relator $r_2$ and inductively in the $k$-th step we draw the relator $r_k$ from the set $W'_n - \{r_1, r_2, \dots, r_{k-1} \} $. Therefore, we obtain a sequence of $\lfloor (2n-1)^{4d} \rfloor$ random variables $r_1, r_2, \dots , r_{\lfloor (2n-1)^{4d} \rfloor }$ and we define $R_n := \{r_1, r_2, \dots, r_{\lfloor (2n-1)^{4d} \rfloor} \}$.

For $1 \leq i \leq |R_n|$ we define a random variable $X_i$: $X_i = 1$ when $r_i \in (W'_n - W_n)$ and $0$ otherwise. Note that $\frac{|W'_n - W_n|}{|W'_n|} < \frac{15}{16}$. Hence for each $1 \leq i \leq |R_n|$ we have $\mathbb{E} X_i < \frac{15}{16}$. Let $S:= \sum_{i=1}^{|R_n|} X_i$. By \cite[Corollary 1.1]{serfling} for $\mu = \frac{|W_n|}{|W'_n|}, a=1, b=1$ and $f_* = 0$ we obtain

$$\mathbb{P}\left(S \geq \lfloor (2n-1)^{4d} \rfloor(\mu + t) \right) < \exp(-2\lfloor (2n-1)^{4d} \rfloor t^2).$$

Let $t=\frac{16}{17} - \mu$. Then

\begin{equation}\label{serfling_pomocniczy}
\mathbb{P}\left(S \geq \lfloor (2n-1)^{4d} \rfloor \frac{16}{17} \right) < \exp(-2 \lfloor (2n-1)^{4d} \rfloor t^2).
\end{equation}

Since $\mu < \frac{15}{16}$ we have $t > \frac{16}{17} - \frac{15}{16}$ and the right hand side of (\ref{serfling_pomocniczy}) converges to 0 when $n \rightarrow \infty$. Observe that $S = |R_n \cap (W'_n - W_n)|$, so the proof of (\ref{lemma:gromov_lemma}) is complete. Consequently,

 \begin{equation}
 \mathbb{P} \left( |R_n \cap W_n|) > \frac{1}{17}\lfloor (2n-1)^{4d} \rfloor \right) \rightarrow_{n \rightarrow \infty} 1.
 \end{equation}

Since for sufficiently large $n$ we have $ \frac{1}{17}\lfloor (2n-1)^{4d}\rfloor > n^{4d'} $ the proof of Lemma \ref{lem:ddp} is complete.
\end{proof}

\begin{proof}[Proof of Theorem \ref{thm:square_triv}]
Let $G$ be a random group in the square model at density $d > \frac{1}{2}$. Choose any $\frac{1}{2} < d' < d$. Lemma \ref{lem:ddp} guarantees that there are at least $n^{4d'}$ positive relators in the presentation of $G$. From the previous section we know that this is sufficient for $G$ to be trivial with overwhelming probability.
\end{proof}


\section{Isoperimetric inequality}

In this chapter we are going to introduce van Kampen diagrams, then prove the ,,isoperimetric inequality'' and discuss its consequences. As we will see it implies (with overwhelming probability) freeness of random groups for densities $\leq \frac{1}{4}$ and hyperbolicity for $\leq \frac{1}{2}$.


The van Kampen diagrams concept is a geometric way to represent how all equalities holding in a group are derived from combinations of relators. The definitions and notations below largely follow \cite[Ch. V.]{ls}.

\begin{defi}\label{def:van_kampen}
(\textit{Van Kampen diagrams}). Let $G=\langle A | R\, \rangle$ be a group presentation where all $r \in R$ are cyclically reduced words in the free group $F(A)$. We will denote this presentation by $(\dagger)$. The alphabet $A$ and the set of defining relations $R$ are often assumed to be finite, which corresponds to a finite group presentation, but this assumption is not necessary for the general definition of a van Kampen diagram. Let $R_*$ be the symmetrized closure of $R$, that is, let $R_*$ be obtained from $R$ by adding all cyclic permutations of elements of $R$ and of their inverses.

A \textit{van Kampen diagram} over the presentation $(\dagger)$ is a planar finite cell complex $\mathcal D\, $, given with a specific embedding $\mathcal D\subseteq \mathbb R^2\,$ with the following additional data and satisfying the following additional properties:

\begin{enumerate}
\item The complex $\mathcal D\,$ is connected and simply connected.
\item Each edge (one-cell) of $\mathcal D\,$ is labelled by an arrow and a letter $a \in A$.
\item Some vertex (zero-cell) which belongs to the topological boundary of $\mathcal D\subseteq \mathbb R^2\,$ is specified as a base-vertex.
\item For each region (two-cell) of $\mathcal D\,$ for every vertex the boundary cycle of that region and for each of the two choices of direction (clockwise or counter-clockwise) the label of the boundary cycle of the region read from that vertex and in that direction is a freely reduced word in $F(A)$ that belongs to $R_*$.

\end{enumerate}
\end{defi}

A van Kampen diagram $\mathcal D\,$ is called \textit{non-reduced} if there exists a \textit{reduction pair} in $\mathcal D\,$, that is, a pair of distinct regions of $\mathcal D\,$ such that their boundary cycles share a common edge and such that their boundary cycles, read starting from that edge, clockwise for one of the regions and counter-clockwise for the other, are equal as words in $A  \cup A^{-1}$. If no such pair of regions exists,  $\mathcal D\,$ is called \textit{reduced}.

An \textit{internal edge} is an edge $e$ such that $\emph{Int(}e\emph{)} \subset \emph{Int(} \mathcal D\, \emph{)}$. An \textit{internal vertex} is a vertex contained in $ \emph{Int(} \mathcal D\, \emph{)}$.

By $|\partial \mathcal{D}|$ we will denote the length of the boundary word of diagram $\mathcal{D}$ and by $|\mathcal{D}|$ the number of faces of $\mathcal{D}$.

\begin{defi}\label{def:annular}
If in the definition above we replace simple connectivity by the assumption that our diagram is homotopically equivalent to an annulus, we get the definition of an \textit{annular diagram}.
\end{defi}

\begin{defi}\label{def:moebius}
If we change two things in Definition \ref{def:van_kampen}:
\begin{enumerate}
\item Replace the assumption that the cell complex is planar and given with a specific embedding into $\mathbb{R}^2$ by the assumption that the cell complex is given with an embedding into the real projective plane
\item Replace simple connectivity by the assumption that our diagram is homotopically equivalent to the M\"{o}bius strip.
\end{enumerate}
we get the definition of a \textit{twisted diagram}.
\end{defi}

An important theorem of Van Kampen states that the boundary words of Van Kampen diagrams are exactly those words which are equal to the identity element in the presentation.

\begin{lemma}\label{lem:parzy}
If in the presentation $\left< A_n | R \right>$ the set of relators $R$ consists of only positive words then for every van Kampen diagram $\mathcal{D}$ with respect to this presentation each internal vertex of $\mathcal{D}$ has even valence.
\end{lemma}

\begin{proof}
Let $v$ be a vertex of $\mathcal{D}$ and denote by $e_1, e_2, \dots, e_k$ consecutive edges of $\mathcal{D}$ ending in $v$. For $1 \leq i \leq k$ let $F_{i}$ be the face containing edges $e_i$ and $e_{i+1}$ (where $e_{k+1} = e_1$). Since the presentation consists of only positive words, for $1 \leq i \leq k-1$ faces $F_i$ and $F_{i+1}$ have the opposite orientation and similarly $F_k$ and $F_1$ have the opposite orientation. Therefore, the number of faces must be even, which implies that $k$ is an even number.
\end{proof}

This lemma will be useful in the proof of Theorem \ref{thm:once}. Now we are able to formulate and prove the following theorem (inspired by \cite[Theorem 13]{oll05}), which we will call ``isoperimetric inequality in the positive square model''

\begin{theo}\label{thm:ie} For any $\varepsilon > 0$, in the positive square model at density $d < \frac{1}{2}$ with overwhelimg probability all reduced van Kampen diagrams associated to the group presentation satisfy

$$|\partial D| \geq 4(1 - 2d -  \varepsilon) |D|.$$
\end{theo}

Note that our model can be considered as a special case of Gromov's density model (where $l$ = 4). Hence, proving Theorem~\ref{thm:ie} we will be mimicking the proof of the analogous theorem in Gromov's model. There are only a few details that have to be changed. Let us recall some definitions and propositions from ``Proof of the density one half theorem'' in \cite{oll05}.

An \textit{abstract diagram} is a van Kampen diagram in which we forget the actual relators associated with the faces, but only remember: the geometry of the diagram, which faces bear the same relator, the orientation, and the beginning points of relators. The $s$-tuple $(\omega_1, \dots, \omega_s)$ of cyclically reduced words is said to \textit{fulfill the abstract diagram} $D$ if there is a van Kampen diagram formed by polygons bearing these words, which after forgetting the relators gives $D$. An abstract diagram is \textit{reduced} if no edge is adjacent to two faces bearing the same relator with opposite orientations such that the edge is the $k$-th edge of both faces for some $k$. See \cite{oll05} (page 83) for details. The following proposition is inspired by \cite[Proposition 58]{oll05}

\begin{prop}\label{prop:full}
Let $R$ be a set of $\lfloor n^{4l} \rfloor$ relators chosen randomly, with uniform distribution, from the set of positive words of length $l$ on $n$ generators. Let $D$ be a reduced abstract diagram and let $\varepsilon > 0$. Then either $|\partial D| \geq |D| l (1 - 2d - 2\varepsilon)$ or the probability that there exists a tuple of relators in $R$ fulfilling $D$ is less than $n^{- \varepsilon l}$.
\end{prop}

We will present the proof of this proposition in a more general form at the end of this chapter. Now we will recall the theorem known as the ''local-global principle'' or the ``Gromov--Cartan--Hadamard Theorem''. This principle has many different formulations. The variant best suited to our context is \cite[Theorem 60]{oll05} which is a slight modification of \cite[Proposition 8]{oll-f}

\begin{theo}\label{thm:lg}
Let $G = \left<a_1, \dots, a_n | R \right>$ be a finite group presentation and let $l_1$ and $l_2$ be the minimal and maximal lengths of relator in $R$. For a van Kampen diagram $D$ with respect to the presentation set, we define:

$$ A(D) = \sum_{\emph{f face of D}} |\partial f|,$$
where $|\partial f|$ is the length of the boundary path of face $f$.

Let $C>0$. Choose $\varepsilon > 0$. Suppose that for some $K$ greater than $10^{50}\left(\frac{l_1}{l_2} \right)^3 \varepsilon^{-2} C^{-3}$ any reduced van Kampen diagram $D$ with $A(D) \leq Kl_2$ satisfies

$$|\partial D| \geq C A(D).$$
Then any reduced van Kampen diagram $D$ satisfies

$$|\partial D| \geq (C - \varepsilon) A(D)$$
and in particular the group is hyperbolic.
\end{theo}

\subsection{Isoperimetric inequality in the positive square model}

Now we are able to provide the proof of Theorem \ref{thm:ie}. Our argumentation is very close to Ollivier's.

\begin{proof}[Proof of Theorem \ref{thm:ie}]
In our case all relators in the presentation have the same length $l=4$, so $A(D) = 4|D|$. In particular, the assumption $A(D) \leq K l_2$ in the Theorem \ref{thm:lg} becomes $|D| \leq K$ i.e. we have to check diagrams with at most $K$ faces.

Choose any $\varepsilon > 0$. Set $C=1-2d-2\varepsilon$ and $K = 10^{50} \varepsilon^{-2}(1 - 2d - 2\varepsilon)^{-3}$. Let $n$ be the number of letters in the generating alphabet. Let $N(K, n)$ be the number of abstract reduced diagrams with at most $K$ square faces. We know from Proposition~\ref{prop:full} that for any fixed reduced abstract diagram $D$ violating the inequality $|\partial D| \geq 4(1 - 2d - 2 \varepsilon)|D|$ the probability that it appears as a van Kampen diagram of the presentation is $\leq n^{-4 \varepsilon}$. So the probability that there exists a reduced van Kampen diagram with at most $K$ faces, violating the inequality is $\leq N(K, n)n^{-4 \varepsilon}$.

Observe that there are finitely many planar diagrams with at most $K$ square faces. There are also finitely many ways to decide which faces would bear the same relator, and also finitely many ways to choose the beginning point of each relator. Therefore, the values $\{N(K,n) \}_{\{n \in \mathbb{N}\} }$ have a uniform bound $M$ (independent of $n$).

Hence, for fixed $\varepsilon$, $\displaystyle \lim_{n \rightarrow \infty} N(K,n)n^{-4 \varepsilon} = 0$. Applying Theorem \ref{thm:lg} (with our choice of $C$ and $\varepsilon$) yields that all reduced van Kampen diagrams $D$ satisfy $|\partial D| \geq 4(1 - 2d - 3\varepsilon)|D|$ as needed.

\end{proof}

\begin{cor}\label{cor:hyp}
In the positive square model at density $d < \frac{1}{2}$ a random group is hyperbolic with overwhelming probability.
\end{cor}

\subsection{A generalization of the isoperimetric inequality}

In this section we consider a more general case where there are some fixed letters in the diagrams and the length of relators in the presentation equals $l$.

\begin{defi}\label{def:fix}
Let $A$ be an abstract diagram. Let $e_1, e_2, \dots, e_K$ be a set of $K$ distinct edges in $A$. We call them \textit{fixed edges (of A)}. Let $a_{i_1}, a_{i_2}, \dots, a_{i_K}$ be a sequence of generators labeling edges $e_1, e_2, \dots, e_K$ successively and according to the orientation of the face or faces containing $e_j$. We call these generators \textit{fixed letters (of A)}. We call such a diagram with some labellings on edges an \textit{abstract diagram with $K$ fixed edges}.

We say that a tuple of relators \textit{fulfills} $A$ if this tuple fulfills $A$ as an abstract diagram and this fulfilling is consistent with labels on fixed edges.
\end{defi}

Our goal is to prove the following statement

\begin{theo}\label{thm:lab} Let R be a set of $\lfloor n^{4l} \rfloor$ relators chosen randomly, with uniform distribution, from the set of positive words of length $l$ on $n$ generators. Let $A$ be an abstract diagram with $K$ fixed edges and let $\varepsilon > 0$. Then either $|\partial A| - 2K \geq |A|l(1 - 2d - \varepsilon)$ or the probability that there exists a tuple of relators in $R$ fulfilling $A$ is $< n^{-\varepsilon l}$.
\end{theo}

Our proof is based on the proof of \cite[Proposition 58]{oll05}. To prove our theorem we need some more definitions. Let $N$ be the number of distinct relators in $A$. For $1 \leq i \leq N$ let $m_i$ be the number of faces bearing relator $i$. Up to reordering the relators we can suppose that $m_1 \geq m_2 \geq \dots \geq m_N$.
	
For $1 \leq i_1, i_2 \leq N$ and $1 \leq k_1, k_2 \leq l$ we say that $(i_1, k_1) > (i_2, k_2)$ if $i_1 > i_2$ or $i_1 = i_2$ but $k_1 > k_2$. Let $e$ be an edge of $A$ adjacent to faces $f_1, f_2$ bearing relators $i_1$ and $i_2$, which is the $k_1$-th edge of $f_1$ and $k_2$-th edge of $f_2$. If $e$ is not a fixed edge we say that $e$ \textit{belongs} to $f_1$  if $(i_1, k_1) > (i_2, k_2)$ and $e$ \textit{belongs} to $f_2$ if $(i_2, k_2) > (i_1, k_1)$. If $e$ is a fixed edge then we say that $e$ \textit{belongs} to both faces: $f_1$ and $f_2$. If $e'$ is a fixed edge which is adjacent to face $f$ and $e'$ is a boundary edge, then we say that $e'$ \textit{belongs} to $f$.

Note that since $A$ is reduced, then each internal edge, which is not a fixed edge, belongs to some face: indeed if $(i_1, k_1) = (i_2, k_2)$ then either the two faces have opposite orientations and then $A$ is not reduced, or they have the same orientations and the diagram is never fulfillable since a letter would be its own inverse.

Let $\delta(f)$ be the number of edges belonging to face $f$. Since each internal edge which is not a fixed edge belongs to some face we have:

$$|\partial A| - 2 K = l|A| - 2 \sum_{\emph{f face of A}} \delta(f).$$

For $1 \leq i \leq N$ let

$$\kappa_i = \max_{f \text{ face bearing relator i}} \delta(f)$$

Then:
\begin{equation}\label{eq:2}
|\partial A| - 2 K \geq l|A| - 2 \sum_{1 \leq i \leq N} m_i \kappa_i
\end{equation}

\begin{lem}\label{lem:59}
For $1 \leq i \leq N$ let $p_i$ be the probability that $i$ randomly chosen positive words $w_1, w_2, \dots, w_i$ partially fulfill $A$ and let $p_0 = 1$. Then
\begin{equation}\label{eq:3}
\frac{p_i}{p_{i-1}} \leq n^{-\kappa_i}.
\end{equation}
\end{lem}

\begin{proof}
Suppose that first $i-1$ words $w_1, \dots, w_{i-1}$ partially fulfilling $A$ are given. Then successively choose the letters of the word $w_i$ in a way to fulfill the diagram. Let $k \leq l$ and suppose that the first $k-1$ letters of $w_i$ are chosen. Let $f$ be the face realizing the maximum of $\kappa_i$ and let $e$ be the $k$-th edge of the face $f$.

If $e$ belongs to $f$ this means that there is another face $f'$ meeting $e$ which bears relator $i' < i$ or bears $i$ too, but $u$ appears in $f'$ as a $k'$-th edge for $k' < k$ or $e$ is a fixed edge. In all these cases the letter on the edge $e$ is imposed by some letter already chosen so drawing it at random has probability $\leq \frac{1}{n}$.

Combining all these observations we get that the probability to choose at random the correct word $w_i$ is at most $p_{i-1}n^{-\kappa_i}$.
\end{proof}

Now we can provide the proof of Theorem \ref{thm:lab}.

\begin{proof}[Proof of Theorem \ref{thm:lab}]
For $1 \leq i \leq n$ let $P_i$ be the probability that there exists a $i$-tuple of words partially fulfilling $A$ in the random set of relators $R$. We trivially have:

\begin{equation}\label{eq:4}
P_i \leq |R|^i p_i = n^{idl}p_i
\end{equation}

Combining Equations (\ref{eq:2}) and (\ref{eq:3}) we get

$$|\partial A| - 2K \geq l|A| + 2 \sum_{i=1}^N m_i(\log_{n}p_i - \log_{n}p_{i-1})$$
$$=l|A| + 2 \sum_{i=1}^{N-1} (m_i - m_{i+1})\log_{n}p_i + 2 m_N \log_{n}p_N - 2 m_1 \log_{n}p_0. $$

Now $p_0 = 1$ so $\log_{n}p_0 = 0$ and we have

$$|\partial A| - 2K \geq l|A| + 2 \sum_{i=1}^{N-1} (m_i - m_{i+1})\log_{n}p_i + 2 m_N \log_{n}p_N $$

Now using (\ref{eq:4}) and again the fact that  $m_{i+1} \leq m_i$ we have

$$|\partial A| - 2K \geq l|A| + 2 \sum_{i=1}^{N-1} (m_i - m_{i+1})(\log_{n}P_i - i d l) + 2 m_N \log_{n}(P_N - N d l)$$

Observe that $\sum_{i=1}^{N-1} (m_i - m_{i+1})i d l + m_N N d l  = dl \sum_{i=1}^N m_i = dl|A|$. Hence

$$|\partial A| - 2K \geq l|A|(1 - 2d) + 2 \sum_{i=1}^{N-1} (m_i - m_{i+1})\log_{n}P_i + 2 m_N \log_{n}P_N$$

Setting $P = \min_i P_i$ and using that $m_i \geq m_{i+1}$ we get

$$|\partial A| - 2K \geq l|A|(1 - 2d) + 2(\log_{n}P) \sum_{i=1}^{N-1} (m_i - m_{i+1}) + 2 m_N \log_{n}P $$
$$= l|A|(1-2d) + 2m_1 \log_{n}P \geq |A|( l(1-2d) + 2\log_{n}P),$$
since $m_1 \leq |A|$. Of course a diagram is fulfillable if it is partially fulfillable for any $i \leq n$ and so

$$\emph{Probability(A is fullfillable by relators of R)} \leq P \leq n^{\frac{1}{2}\left(\frac{|\partial A| - 2K}{|A|} - l(1-2d)\right)},$$
which was to be proven.
\end{proof}

Proposition \ref{prop:full} results as a special case of Theorem \ref{thm:lab} when there are no fixed letters.

\subsection{Isoperimetric inequality in the square model}\label{sec:iso}

In this section we are going to prove that all previous results remain true in the square model. Firstly, we will show the following:

\begin{theo}\label{thm:labs} Let R be a set of $\lfloor (2n-1)^{d l} \rfloor$ relators chosen randomly, with uniform distribution, from the set of cyclically reduced words of length $l$ on $n$ generators. Let $A$ be an abstract diagram with $K$ fixed edges and let $\varepsilon > 0$. Then either $|\partial A| - 2K \geq |A|l(1 - 2d - \varepsilon)$ or the probability that there exists a tuple of relators in $R$ fulfilling $A$ is $< (2n-1)^{-\varepsilon l}$.
\end{theo}

Actually the proof of this theorem goes completely analogically to the proof of Theorem \ref{thm:lab}. Again, let $N$ be the number of distinct relators in $A$, and for $1 \leq i \leq N$ let $m_i$ be the number of faces bearing relator $i$. Up to reordering the relators we can suppose that $m_1 \geq m_2 \geq \dots \geq m_N$. Again, let $\delta(f)$ be the number of edges belonging to face $f$ and, for $1 \leq i \leq N$ we define:

$$\kappa_i = \max_{f \text{ face bearing relator i}} \delta(f) $$

Then we have (as in the previous section)

\begin{equation}
|\partial A| - 2K \geq l|A| - 2 \sum_{1 \leq i \leq N}m_i \kappa_i
\end{equation}

\begin{lem}\label{lem:59s}
For $1 \leq i \leq N$ let $p_i$ be the probability that $i$ randomly chosen cyclically reduced words $w_1, w_2, \dots, w_i$ partially fulfill $A$ and let $p_0 = 1$. Then
\begin{equation}\label{eq:3s}
\frac{p_i}{p_{i-1}} \leq (2n-1)^{-\kappa_i}.
\end{equation}
\end{lem}

\begin{proof}
To prove this lemma one has to change only one thing in the proof of Lemma \ref{lem:59}: replace $n$ with $2n-1$.
\end{proof}

Now we will provide the proof of Theorem \ref{thm:labs}.

\begin{proof}[Proof of Theorem \ref{thm:labs}.] For $1 \leq i \leq N$ let $P_i$ be the probability that there exists an $i$-tuple of words partially fulfilling $A$ in the random set of relators. We have (analogically to (\ref{eq:4})):

\begin{equation}
P_i \leq |R|^i p_i \leq (2n-1)^{dl} p_i
\end{equation}

Repeating the reasoning in the proof of Theorem \ref{thm:lab} but replacing $n$ with $2n-1$ we obtain

$$\emph{Probability(A is fullfillable by relators of R)} \leq (2n-1)^{\frac{1}{2}\left(\frac{|\partial A| - 2K}{|A|} - l(1-2d)\right)},$$
which was to be proven.
\end{proof}

Now we are able to prove the isoperimetric inequality for the square model:

\begin{theo}\label{thm:ies} For any $\varepsilon > 0$, in the square model at density $d < \frac{1}{2}$ with overwhelimg probability all reduced van Kampen diagrams associated to the group presentation satisfy:

$$|\partial D| \geq 4(1 - 2d -  \varepsilon) |D|.$$
\end{theo}

\begin{proof}
The proof is completely analogous to the proof of Theorem \ref{thm:ie}: the only change is that we have to use the Theorem \ref{thm:labs}, instead of Proposition \ref{prop:full}, to obtain that for any fixed reduced abstract diagram $D$ violating the inequality $|\partial D| \geq 4(1 - 2d - 2 \varepsilon)|D|$ the probability that it appears as a van Kampen diagram of the presentation is $\leq (2n-1)^{-4 \varepsilon}$, which is $< n^{-4 \varepsilon}$.
\end{proof}

\begin{cor}\label{cor:hyps}
In the square model at density $d < \frac{1}{2}$ a random group is hyperbolic with overwhelming probability.
\end{cor}


\section{Freeness}

In this chapter we are going to consider the case where density of relations is small. Our goal is to prove the following statement:

\begin{theo}[freeness theorem]\label{thm:free}
In the positive square model at density $d < \frac{1}{4}$ a random group is free with overwhelming probability.
\end{theo}

To provide the proof of this theorem we will introduce several geometric objects:

\begin{defi}\label{def:complex}
A \textit{square complex} is a metric polyhedral complex in which each cell is isometric to the Euclidean square $[-\frac{1}{2}, \frac{1}{2}]^2$ and the gluing maps are isometries.
\end{defi}

Observe that we allow to glue a cell to itself and to glue two cells several times along distinct pairs of faces. Notice that van Kampen diagrams in the positive square model are square complexes.

Now we will introduce one of the basic notions in geometric group theory: hypergraphs. The following definitions are taken from \cite[Definition 2.1]{ow11}.

\begin{defi}\label{def:hyp}
Let $X$ be a connected square complex. We define a graph $\Gamma$ as follows: The set of vertices of $\Gamma$ is the set of $1$-cells of $X$. There is an edge in $\Gamma$ between two vertices if there is some $2$-cell $R$ of $X$ such that these vertices correspond to opposite $1$-cells in the boundary of $R$ (if there are several such $2$-cells we put as many edges in $\Gamma$). The $2$-cell $R$ is the $2$-cell of $X$ \textit{containing} the edge.

There is a natural map $\varphi$ from $\Gamma$ to $X$, which sends each vertex of $\Gamma$ to the midpoint of the corresponding $1$-cell of $X$ and each edge of $\Gamma$ to a segment joining two opposite points in the $2$-cell $R$. Note that the images of two edges contained in the same $2$-cell $R$ always intersect, so that in general $\varphi$ is not an embedding.

A \textit{hypergraph} in $X$ is a connected component of $\Gamma$. The $1$-cells of $X$ through which a hypergraph passes are \textit{dual} to it. The hypergraph $\Lambda$ \textit{embeds} if $\varphi$ is an embedding from $\Lambda$ to $X$, that is, if no two distinct edges of $\Lambda$ are mapped to the same $2$-cell of $X$.

We call the subdiagram of $X$ consisting of all open faces containing edges of the hypergraph $\Lambda$ the \textit{carrier of} $\Lambda$.
\end{defi}

The next definition we will give is one of the typical ways to construct a topological space $X$ such that $\pi_1(X) = G$ for given a finitely generated group $G$,.

\begin{defi}[presentation complex]\label{def:pres_complex}
Let $G = \left< A_n | R \right>$ be a group generated by $n$ elements. Consider a bouquet of $n$ circles labeled with elements of $A_n$. For every relator $r \in R$ there is a polygon with as many edges as letters in $r$, which is glued to the bouquet in the following way: the edge labeled by $a \in A_n$ is glued to the circle with label $a$ respecting the orientation. For a random group in the square or positive square model this construction results in a square complex, which we call the \textit{presentation complex}.
\end{defi}

There is a natural map from any van Kampen diagram to the presentation complex of $G$. One of the main steps in our proof of Theorem \ref{thm:free} is the following statement:

\begin{theo}\label{thm:ftree}
In the positive square model at density $d < \frac{1}{4}$ with overwhelming probability all hypergraphs in the presentation complex are embedded trees.
\end{theo}

\begin{proof}
Denote the presentation complex of a random group by $X$. We will estimate the probability $P$ of drawing the set of relators $R_n$ for which the statement does not hold, i.e. there exists a hypergraph in $X$ which is not an embedded tree.

Hence, let us assume that such a hypergraph exists and call it $\Lambda$. The image of $\Lambda$ under the natural map is not a tree, so $\Lambda$ contains a \textit{circuit} (an edge path $(e_1, e_2, \dots, e_k )$ in $\Lambda$ such that images of $e_1$ and $e_k$ intersect in $X$). Without loss of generality we can assume that $k$ is the minimal possible length of a circuit. For $i \in \{1, \dots, k \}$, let $F_i$ be the $2$-cell of $X$ containing the edge $e_i$. We have chosen the circuit of the minimal length, so $F_1 = F_k$ and $F_i \neq F_j$ for $i < j$, except where $i=1, j=k$.

Let $D$ be a diagram consisting of faces $F_1, \dots, F_k$ glued in the following way: for $1 \leq i \leq k-1$ faces $F_i$ and $F_{i+1}$ are glued along $1$-cells which contain the common vertex of $e_i$ and $e_{i+1}$. It can be easily seen that $D$ is either: an annular diagram, a twisted diagram or a van Kampen diagram. We will estimate a probability $P_k$ of drawing the set of relators which allows to construct a diagram with $k$ faces and exactly $3k$ edges and consisting of distinct relators ($D$ has these properties).

Let $E$ be the abstract diagram obtained from $D$. There are $n^{3k}$ $k-$tuples of relators fulfilling $E$. Denote by $L$ the set of these $k-$tuples. To fulfill $E$ one of the elements $\alpha \in L$ must be a subset of the set of drawn words $R_n$.

For $\alpha \in L$ let $P_{\alpha}$ be the probability of drawing the set $R_n$ which contains $\alpha$. Then:

$$P_{\alpha} = \frac{\left(n^4 - k \atop \lfloor n^{ 4d }\rfloor - k \right)}{\left( n^4 \atop \lfloor n^{ 4d }\rfloor \right) }$$
Hence:

\begin{equation}\label{poz}
P_{\alpha} = \frac{( \lfloor n^{ 4d }\rfloor )(\lfloor n^{ 4d }\rfloor - 1) \dots (\lfloor n^{ 4d }\rfloor - k + 1)   }{n^4 (n^4 - 1) \dots (n^4 - k +1)} < \frac{n^{k 4 d}}{n^{4k}}
\end{equation}
We have $P_k \leq n^{3k} P_{\alpha}$. To estimate the probability of the existence of a hypergraph, that is not an embedded tree, we sum $P_k$ over all possible $k$:

\begin{equation} \label{eq:free}
P = \sum_{k=1}^{n^4} P_k \leq \sum_{k=1}^{n^4} n^{3k} \frac{n^{k 4d }}{n^{4k}} < \sum_{k=1}^{\infty} n^{(4d  - 1)k} < \frac{1}{1 - n^{4d - 1}} - 1
\end{equation}

We assumed that $d < \frac{1}{4}$, so the right hand side of (\ref{eq:free}) converges to 0 when $n \rightarrow \infty$.

\end{proof}

Let us recall one of the applications of the HNN extension construction:

\begin{theo}[{\cite[Proposition 1.2]{sw77}}]\label{thm:HNN}
Let $V$ be a Hausdorff topological space and let $Y_1, Y_2 \subset V$ be two distinct, simply-connected and path-connected subsets such that there is a homeomorphism $f: Y_1 \rightarrow Y_2$. By $V*_{f}$ we denote the topological space $V /\sim $, where the relation $\sim $ is defined as follows: $y \sim  f(y)$ for all $y \in Y_1$. Then $\pi_1(V*_{f}) = \pi_1(V)*\mathbb{Z}$.
\end{theo}

Now we are ready to provide the proof of the freeness theorem.

\begin{proof}[Proof of Theorem \ref{thm:free}]
From Theorem \ref{thm:ftree} we know that with overwhelming probability all hypergraphs in the presentation complex $X$ are embedded trees. Let us take an arbitrary hypergraph $\Lambda$. Let $H$ be the carrier of $\Lambda$.

Let us consider the complex $\overline{X - \Lambda}$ (by $\overline{A}$ we denote the completion of the complex $A$ in the path metric). Note that $\overline{X - \Lambda} - (X - \Lambda)$ consists of two isometric copies of $\Lambda$ denoted: $\Lambda_1$ and $\Lambda_2$. Let $\phi : \Lambda_1 \rightarrow \Lambda_2$ be the homeomorphism between $\Lambda_1$ and  $\Lambda_2$. $X - \Lambda$ is homotopically equivalent to $\overline{X - \Lambda}$ and $X - H$. Moreover the space $(\overline{X - \Lambda})*_{\phi}$ is equal to the complex $X$. Hence $\pi_1(X) = \pi_1((\overline{X - \Lambda})*_{\phi})$, and by Theorem \ref{thm:HNN} we obtain $\pi_1(X) = \pi_1(X - H)*\mathbb{Z}$.

We now perform the same procedure for the subcomplex $X_1 := X - H$. We choose an arbitrary hypergraph in $X_1$ and remove its carrier from $X_1$ obtaining a smaller complex $X_2$. By Theorem \ref{thm:HNN}: $\pi_1(X) = \pi_1(X_1) * \mathbb{Z} = \pi_1(X_2) * \mathbb{Z}* \mathbb{Z}$.

We now inductively repeat this procedure. Note that the presentation complex is finite and each time we remove at least one cell, so this induction must stop after a finite number $m$ of steps. Let $X_m$ be the subcomplex obtained after $m$ steps. We cannot perform this procedure on $X_m$ which specifically means that there are no hypergraphs in $X_m$. But the only square complex with no hypergraphs is the square complex consisting of one vertex, which has the trivial fundamental group. Therefore $\pi_1(X) = \pi_1(X_m) * \underbrace{\mathbb{Z} * \dots * \mathbb{Z}}_{m} = \underbrace{\mathbb{Z} * \dots * \mathbb{Z}}_{m}$. Hence $\pi_1(X)$ is a free group with overwhelming probability.
\end{proof}

\subsection{Freeness in the square model}

First, we will prove the following

\begin{theo}\label{thm:ftrees}
In the square model at density $d < \frac{1}{4}$ with overwhelming probability all hypergraphs in the presentation complex are embedded trees.
\end{theo}

\begin{proof}
As in the proof of Theorem \ref{thm:ftree} we only need to prove that with overwhelming probability there are no diagrams with $k$ faces and exactly $3k$ edges in the presentation complex for any $k$. Let $E$ be such an abstract diagram with the minimal number of faces. We will estimate the probability $P_k$ of drawing the set of relators such that $E$ can be fulfilled. There are at most $(2n)^{k-1}(2n-1)^{2k-1}(2n-2)^2$ $k-$tuples of relators fulfilling $E$. Note that for any $\delta_1$ this number is smaller than $(2n-1)^{(3+\delta_1)k}$ for a sufficiently large $n$. Denote by $L$ the set of these $k-$tuples. To fulfill $E$ one of the $k$-tuples $\alpha \in L$ must be a subset of the set of drawn words $R_n$.

For $\alpha \in L$ let $P_{\alpha}$ be the probability of drawing the set $R_n$ which contains $\alpha$. Then

$$P_{\alpha} = \frac{\left(|W'_n| - k \atop \lfloor (2n-1)^{ 4d }\rfloor - k \right)}{\left( |W'_n| \atop \lfloor (2n-1)^{ 4d }\rfloor \right) }$$
Hence

\begin{equation}\label{eq:pozs}
P_{\alpha} = \frac{( \lfloor (2n-1)^{ 4d }\rfloor )(\lfloor (2n-1)^{ 4d }\rfloor - 1) \dots (\lfloor (2n-1)^{ 4d }\rfloor - k)   }{|W'_n| (|W'_n| - 1) \dots (|W'_n| - k)} < \frac{(2n-1)^{k 4 d}}{|W'_n|^{k}}
\end{equation}

Note that $|W'_n| \geq 2n(2n-1)^2(2n-2)$ which, for any $\delta_2 > 0$ is greater than $(2n-1)^{4 - \delta_2}$ for a sufficiently large $n$. Therefore, we can estimate the right hand side of (\ref{eq:pozs}) by  $(2n-1)^{(4d - 4 + \delta_2)k}$. As in the proof of Theorem \ref{thm:ftree} we estimate the sum of $P_k$ over all possible $k$

\begin{equation} \label{eq:frees}
P = \sum_{k=1}^{|W'_n|} P_k < \sum_{k=1}^{\infty} (2n-1)^{(4d - 1 + \delta_1 + \delta_2)k} <  \frac{1}{1 - (2n-1)^{4d - 1 + \delta_1 + \delta_2}}.
\end{equation}

We assumed that $d < \frac{1}{4}$, so we can choose $\delta_1 > 0$ and $\delta_2 > 0$, such that $4d - 1 + \delta_1 + \delta_2 < 0$. Then the right hand side of (\ref{eq:frees}) converges to 0 when $n \rightarrow \infty$.
\end{proof}

\begin{theo}[freeness theorem in the square model]\label{thm:frees}
In the square model at density $d < \frac{1}{4}$ a random group is free with overwhelming probability.
\end{theo}

\begin{proof}
The proof is identical to the proof of Theorem \ref{thm:free} with the only one change: we use now Theorem \ref{thm:ftrees} instead of Theorem \ref{thm:ftree}.
\end{proof}


\section{Groups without Property (T)}

In this chapter our goal is to prove the following statement

\begin{theo}\label{thm:nT}
In the positive square model at density $d < \frac{1}{3}$ with overwhelming probability a random group does not have property \emph{(T)}.
\end{theo}

First, we will define \textit{Kazhdan's property} (T) and formulate some basic facts. Here we are only concerned with discrete, finitely generated groups. For a more complete treatment including property (T) see \cite{bdl}.
	Let  $\Gamma$ be a finitely generated group with a finite generating set $S$. Let $H$ be a Hilbert space and $\pi : \Gamma \rightarrow \mathbb{B}(H)$ a unitary representation of $\Gamma$ on $H$. We will say that $\pi$ has \textit{almost invariant vectors} if for every $\varepsilon > 0$ there exists $u_{\varepsilon} \in H$ such that for every $s \in S$ we have $||\pi(s)u_{\varepsilon} - u_{\varepsilon}|| < \varepsilon ||u_{\varepsilon}||$. A vector $u \in H$ is called \textit{invariant} if $\pi(g)u = u$ for every $g \in \Gamma$.
	
\begin{defi}\label{def:T}
We say that $\Gamma$ has \textit{property} (T) if for every $H$ and $\pi$ the following holds:
if $\pi$ has almost invariant vectors, then $\pi$ has an invariant vector.
\end{defi}

For our purpose the following criterion will be useful

\begin{theo}[\cite{nr98}]\label{thm:criterion}
If a group $G$ has a subgroup $H$ with the number of relative ends at least 2 then $G$ does not have Property (T).
\end{theo}

We will be mimicking the proof of the analogous theorem in Gromov's model which states that for densities $< \frac{1}{5}$ a random group in the Gromov density model does not have property (T) with overwhelming probability \cite{ow11}.

Until the end of this chapter let $G$ be a random group in the positive square model and $\cay$ its Cayley complex, that is, the universal cover of the presentation complex.

\subsection{Hypergraphs in the Cayley complex are embedded trees}

\begin{lem}\label{lem:et} In the positive square model for densities $< \frac{1}{3}$ the hypergraphs in the Cayley complex of a random group are embedded trees.
\end{lem}

To provide the proof we need a notion of a collared diagram which was introduced by Ollivier and Wise to investigate hypergraphs in the Gromov model.

\begin{defi}
We say that a reduced van Kampen diagram $D$ is a \textit{collared diagram} if there is a vertex $v$ in the boundary such that for every other boundary vertex there is exactly one internal edge which ends in this vertex. Moreover, we assume that $v$ is the end of exactly $0$, $1$ or maximally $2$ internal edges.

Let us denote this set of internal edges by $L$. Let $\lambda \subset D$ be the hypergraph segment consisting of all edges dual to the elements of $L$.

If there is exactly one internal edge ending in $v$ we say that a diagram is \textit{cornerless}. In this case it can be easily seen that $\lambda$ is a circuit.

If the diagram is collared and not cornerless then $\lambda$ is not a loop, but there is a $2$-cell called a \textit{corner} which contains two edges of $\lambda$.

Moreover, there is a natural combinatorial map $\varphi : D \rightarrow \cay$ such that the image $\varphi(\lambda)$ is a hypergraf segment in $\cay$. For such $\lambda$ we say that $D$ is \textit{collared} by segment $\varphi(\lambda)$.  The definition is illustrated in Figure \ref{fig:col}.
\end{defi}

\begin{figure}[h!]
	\centering
		\includegraphics[scale=0.3]{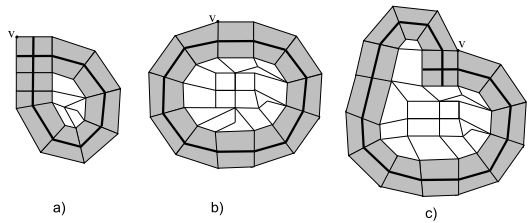}
	\caption{Collared diagrams. The segment $\lambda$ is drawn with a thick line.}
	\label{fig:col}
\end{figure}

In \cite[Definition 3.11]{ow11} Ollivier and Wise defined diagrams collared by hypergraphs and paths for an arbitrary length $l$ of relators. Their definition coincides with ours for $l=4$, the number of collaring hypergraphs equal 1 and the number of collaring paths equal 0. The following theorem shows the relation between collared diagrams and hypergraphs:

\begin{theo}[{\cite[Theorem 3.5]{ow11}}]\label{thm:collar}
Let $\Lambda$ be some hypergraph in $\cay$. The following conditions are equivalent:
\begin{enumerate}
\item $\Lambda$ is an embedded tree.
\item There is no collared diagram collared by a segment of $\Lambda$.
\end{enumerate}
\end{theo}

\begin{proof}[Proof of Lemma \ref{lem:et}]
Assume on the contrary that there is a hypergraph which is not an embedded tree. From Theorem \ref{thm:collar} we know that there is a diagram $D$ collared by some segment $\lambda$. Let $n= |\partial D|$.  For $\varepsilon < 2\left(\frac{1}{3} - d \right)$ from Theorem \ref{thm:ie} (isoperimetric inequality) we have

\begin{equation}\label{nier1}
n=|\partial D| \geq 4|D|(1-2d-\varepsilon) > \frac{4}{3}|D|
\end{equation}

We have two possibilities: either $D$ is cornerless or not. First consider the case where $D$ is cornerless. Then $|D| \geq n$. From (\ref{nier1}) we know that with overwhelming probability all collared cornerless diagrams satisfy:

\begin{equation}\label{nier2}
n > \frac{4}{3}n,
\end{equation}
which is a contradiction. Therefore, with overwhelming probability there are no such diagrams.

Let us now consider the case where the diagram $D$ is not cornerless. Then $|D| \geq n-1$. We have two possibilities $|D| \geq n$ or $|D|=n-1$. If $|D| \geq n$ we again obtain (\ref{nier2}), which is a contradiction. Therefore, with overwhelming probability there are no such diagrams. The only remaining case is where $|\partial D| = n$ and $|D| = n-1$. Again we use (\ref{nier1}) to obtain:

$$n > \frac{4}{3}(n-1).$$

It can be easily seen that for $n > 3$ this is not possible. So we only have to exclude the diagram $|D| = 2, |\partial D|=3$. But there are no diagrams with odd boundary length.
\end{proof}

\begin{lem}[{\cite[Lemma 2.3]{ow11}}]\label{lem:con}
Suppose a hypergraph $\Lambda$ is an embedded tree in $\cay$. Then $\cay - \Lambda$ consists of two connected components.
\end{lem}

\begin{proof}
This follows easily from the fact that $H_1(\cay) = 0$ and from a Mayer-Vietoris sequence argument applied to the complement of the hypergraph and a neighberhood of the hypergraph.
\end{proof}

\subsection{Hypergraphs are quasi-isometrically embedded}

Now we are going to prove that hypergraphs are quasi-isometrically embedded trees. To do that we need to generalize the notion of a collared diagram.

\begin{defi}[Diagram collared by segment and path]
Let $D$ be a reduced van Kampen diagram and let $x_1, \dots, x_n$ be all the vertices on its boundary in that order. Suppose that for some $2 \leq i \leq n-2$ the following holds: for every $i+1 \leq k \leq n$ there is exactly one internal edge $e_k$ ending in $x_k$. Moreover, we assume that there are no internal edges ending in $x_1$ and $x_i$

It can be easily seen that there exists a hypergraph segment $\lambda$ in $D$ such that edges $e_k$ are dual to $\lambda$ and edges $x_1x_2$ and $x_{i-1}x_i$ are also dual to $\lambda$. Define path $\gamma = (x_1, x_2, \dots, x_i )$. There is a natural combinatorial map $\varphi : D \rightarrow \cay$ such that $\varphi(\lambda)$ is a hypergraph segment in $\cay$ and $\varphi(\gamma)$ is a path in $\cay^1$ joining $\varphi(x_1)$ and $\varphi(x_i)$. In such a case we say that $D$ is \textit{collared by segment $\varphi(\lambda)$ and path $\varphi(\gamma)$}.
\end{defi}

\begin{figure}[h!]
	\centering
		\includegraphics[scale=0.6]{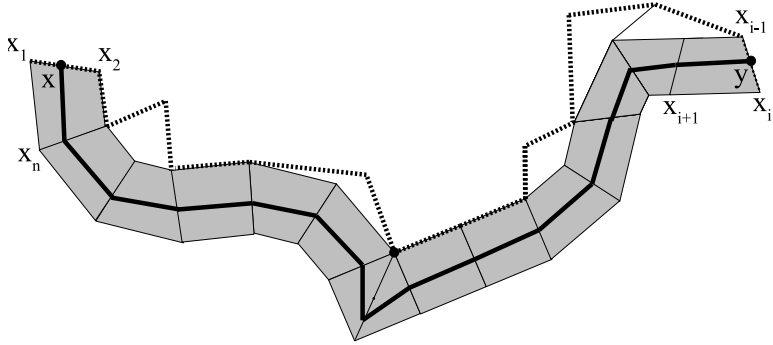}
	\caption{Diagram collared by hypergraph and path. The hypergraph is drawn with a solid thick line and the path with a dotted thick line.}
	\label{fig:col_c}
\end{figure}

Our definition coincides with the one given by Ollivier and Wise in \cite[Definition 3.11]{ow11} for the length of relators equal 4, the number of collaring segments equal 1 and the number of of collaring paths equal 1.

\begin{lem}[{\cite[Lemma 3.17]{ow11}}]\label{thm:cp}
Let $\Lambda$ be a hypergraph that is an embedded tree in $\cay$. Let $\lambda$ be a segment of $\Lambda$. Let $\gamma$ be an embedded path in $\cay^1$ such that the first and the last edge of $\gamma$ are dual to $\lambda$, but $\gamma$ does not intersect $\Lambda$ anywhere except at the first and the last edge. Then there is a diagram collared by the segment $\lambda$ and path $\gamma$.
\end{lem}

\begin{lem}\label{lem:qi} Let $\Lambda$ be a hypergraph in $\cay$ and let $x, y$ be its vertices. By $d_{\Lambda}(x,y)$ we denote the length of the shortest edge path in $\Lambda$ joining $x$ and $y$. By $d_{\cay^1}(x,y)$ we denote the distance in $\cay^1$ between $x$ and $y$. With overwhelming probability, in the positive square model, for every hypergraph $\Lambda$ and every pair of its vertices $x,y$ we have
\begin{equation}\label{eq:iq}
d_{\cay^1}(x,y) \geq \frac{1}{12}d_{\Lambda}(x,y).
\end{equation}
\end{lem}

\begin{proof}
Let $\Lambda$ be a hypergraph in $\cay$ and $x, y$ be vertices of $\Lambda$. Let $\lambda$ be a geodesic in $\Lambda$ joining $x$ and $y$. Let $\gamma$ be a geodesic in $\cay^1$ joining $x$ and $y$.  Let $x_0$ be that end of the edge containing $x$ which does not belong to $\gamma$, and define $y_0$ analogously. Consider the path $\gamma_0 = x_0 x \cup \gamma \cup y y_0$.

Notice that it is sufficient to prove the statement under the additional hypothesis that $\gamma$ intersects $\lambda$ only in $x$ and $y$. If this is not the case we can cut our path into smaller pieces such that each piece intersects $\lambda$ in exactly two points. The inequality in the statement is additive under summing such pieces.

From Lemma \ref{lem:et} we know that with overwhelming probability hypergraphs are embedded trees. Therefore, by Lemma \ref{thm:cp} there exists a diagram $E$ collared by $\lambda$ and $\gamma_0$.

Denote $k = |\lambda|$. Let $C$ be the carrier of $\lambda$. Note that the boundary of $E$ consists of two paths in $\cay^1$ joining $x$ and $y$: we denote them by $\gamma$ and $\psi$. Note that $\psi \subset \partial C$. Therefore, $|\psi| = k+1$. From Theorem \ref{thm:ie} (isoperimetric inequality) we have that with overwhelming probability:

$$|\gamma| + |\psi| > \frac{4}{3}k $$

Since $|\psi| = k+1$ we have

$$|\gamma| + k + 1 >\frac{4}{3}k,$$

which is equivalent to

$$|\gamma| > \frac{1}{3}k - 1$$

If $k \geq 4$, then $\frac{1}{3}k - 1 \geq \frac{1}{12}k$ and we are done. Therefore, the remaining cases are $k=1$, $k=2$, $k=3$. Since $|\gamma| \geq 1$, each of these cases satisfies the inequality (\ref{eq:iq}).
\end{proof}

\begin{cor}\label{cor:sub}
In the positive square model for random groups at density $d \leq \frac{1}{3}$ with overwhelming probability the stabilizer of any hypergraph is a free, quasiconvex subgruop.
\end{cor}

\begin{proof}
Since hypergraphs are trees in $\cay$, their stabilizers act freely on a tree. Groups acting freely on trees are free \cite[Theorem 4]{serr}. Moreover, random groups for densities $< \frac{1}{2}$ are hyperbolic. Now a quasi-isometrically embedded space in a hyperbolic space is quasiconvex.
\end{proof}

\subsection{Pair of hypergraphs which intersect only once}

We now introduce a new type of diagram:

\begin{defi}[Diagram collared by two segments]
Let $D$ be a reduced van Kampen diagram $D$ and let $x_1, \dots, x_n$ be all the vertices on its boundary in that order. Suppose that for some $2 \leq i \leq n-2$ the following holds: for every $k \in  \{2, \dots, i-1 \} \cup \{ i+1, \dots, n\}$ there is exactly one internal edge $e_k$ ending in $x_k$. Moreover, we assume that for $v \in \{x_1, x_i \}$ there are exactly 0 or 2 internal edges ending in $v$.

It can be easily seen that there are two hypergraph segments $\lambda_1$, $\lambda_2$ in $D$ such that for $k \in \{2, \dots, i-1 \}$ edges $e_k$ are dual to $\lambda_1$ and for $k \in \{i+1, \dots, n\}$ edges $e_k$ are dual to $\lambda_2$ and internal edges ending in $x_1$ and $x_i$ are dual to both segments. There are exactly two cells containing edges of both segments $\lambda_1$, $\lambda_2$, called \textit{corners}. There is a natural combinatorial map $\varphi : D \rightarrow \cay$ such that $\varphi(\lambda_1)$ and $\varphi(\lambda_2)$ are hypergraph segments in $\cay$. In such a case we say that $D$ is \textit{collared by segments $\varphi(\lambda_1)$ and $\varphi(\lambda_2)$}.
\end{defi}

\begin{figure}[h!]
	\centering
		\includegraphics[scale=0.5]{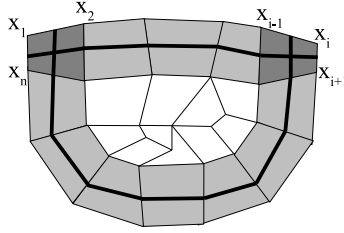}
	\caption{Diagram collared by two segments. The segments are drawn with a solid thick line and the corners are highlighted in dark-gray.}
	\label{fig:diagi}
\end{figure}

Our definition coincides with the one given by Ollivier and Wise in \cite[Definition 3.11]{ow11} for the length of relators equal 4, the number of collaring segments equal 2 and the number of collaring paths equal 0.

\begin{lem}[{\cite[Lemma 3.12]{ow11}}]\label{thm:segments}
Let $\Lambda_1$ and $\Lambda_2$ be two distinct hypergraphs in $\cay$ that are embedded trees. There is more than one point in $\Lambda_1 \cap \Lambda_2$ if and only if there exists a reduced diagram $E$ collared by segments of $\Lambda_1$ and $\Lambda_2$. Moreover, if $\Lambda_1$ and $\Lambda_2$ cross at a 2-cell $C$ we can choose $E$ so that $C$ is one of these corners.
\end{lem}

\begin{theo}\label{thm:once}
With overwhelimg probability, in the positive sqaure model, there exists a pair o hypergraphs $\Lambda_1$, $\Lambda_2$ in $\cay$ such that $\Lambda_1 \cap \Lambda_2$ intersect only once.
\end{theo}

\begin{proof}
Let us consider a pair of hypergraphs $\Lambda$, $\Lambda'$ intersecting at least two times. Assume that $\Lambda$ and $\Lambda'$ cross at a 2-cell $C$. From Theorem \ref{thm:segments} we know that there exists a diagram $E$ collared by segments $\lambda_1 \subset \Lambda_1$ and $\lambda' \subset \Lambda'$ such that $C$ is its corner. Note that $|\partial E| = |\lambda| + |\lambda'|$ and $|E| \geq |\lambda| + |\lambda'| - 2$. From Theorem \ref{thm:ie} (isoperimetric inequality) we obtain that with overwhelming probability:

\begin{equation}\label{eq:trzy}
|\lambda| + |\lambda'| = |\partial E| \geq \frac{4}{3} |E| \geq \frac{4}{3}(|\lambda| + |\lambda'| - 2),
\end{equation}
which is equivalent to $|\lambda| + |\lambda'| < 8$, which implies $|\partial E| \leq 6$. Moreover $|E| \geq |\partial E| - 2$. Therefore, all possibilities which do not violate (\ref{eq:trzy}) are: $|E|=4$, $|\partial E|=6$ and $|E|=2, |\partial E|=4$. There are only three 2-collared diagrams satisfying (\ref{eq:trzy}) (See Figure \ref{fig:diagi})

\begin{figure}[h!]
	\centering
		\includegraphics[scale=0.8]{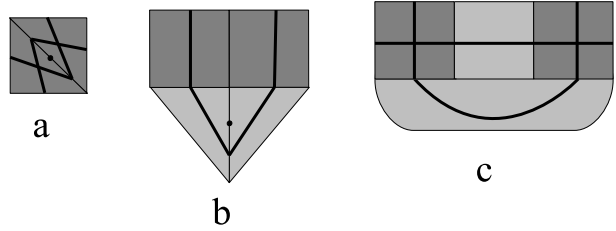}
	\caption{Diagrams collared by two segments. The segments are drawn with a solid thick line and the corners are highlighted in dark-gray.}
	\label{fig:diagi}
\end{figure}

According to Lemma \ref{lem:parzy} diagram $c$ cannot be fulfilled by positive relators. Therefore, we have to consider only cases: $a$ and $b$.

For each 2-cell $O \in \cay$ consider two hypergraphs passing through $O$. Assume, contrary to our conclusion, that each such pair of hypergraphs has at least two points in the intersection. According to our previous discussion this means that for each relator $r$ in the presentation there is a van Kampen diagram $E$ of form $a$ or $b$ (see Figure \ref{fig:diagi}) such that its corner bears $r$.

We can draw relators in two steps: in the first step we draw one relator $r$ and in the second step we draw $\lfloor n^{4d} \rfloor - 1$ remaining relators from the set $W_n - r$. This way of drawing the presentation gives us a specific relator $r$. We will show that with overwhelming probability relator $r$ is not borne by a corner in any van Kampen diagram of the shape $a$ or $b$.

The probability that there exists a van Kampen diagram of type $a$ or $b$ such that its corner bears $r$ is the same as the probability that one of the abstract diagrams presented in Figure \ref{fig:dlab} (where $x, y$ are two consecutive edges of $r$) can be fulfilled by the tuple from the random set of relators.

\begin{figure}[h!]
	\centering
		\includegraphics[scale=0.8]{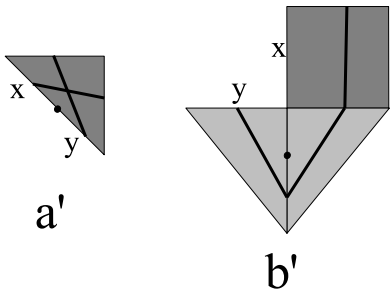}
	\caption{Diagrams with fixed letters. $x$ and $y$ are two consecutive letters of $r$.}
	\label{fig:dlab}
\end{figure}

Note that with overwhelming probability relator $r$ consists of four distinct letters (in fact this probability equals $\frac{n(n-1)(n-2)(n-3)}{n^4}$).

To fulfill $E$ we can use $r$ and other relators. Observe that since with overwhelming probability $r$ consists of different letters and $r$ is a positive word, than two faces bearing relator $r$ cannot be adjacent (such a pair of faces would make a reduction pair).

Hence, diagram $a'$ can be fulfilled only by a relator different than $r$. Let $P_{a'}$ be the probability of fulfilling $a'$. Observe that $a'$ is an abstract diagram with two fixed letters and one face. Moreover, note that $|\partial a'| - 2 \cdot 2 < \frac{4}{3}|a'|$ so from Theorem \ref{thm:lab} (used for $l=4$) the probability of fulfilling $a'$ is less than $n^{- 4 \varepsilon}$ for any $\varepsilon < 2(\frac{1}{3} - d)$. Let us fix some  $\varepsilon < 2(\frac{1}{3} - d)$.

According to the previous observation about faces bearing $r$, there can be maximally two faces bearing $r$ in $b$ in order to fulfill the diagram. Hence, there can be maximally one face bearing $r$ in $b'$.

Let $P_{b'}$ be the probability of fulfilling $b'$ without using the relator $r$. The diagram $b'$ is an abstract diagram with two fixed letters and satisfies $|\partial b'| - 2 \cdot 2 < \frac{4}{3}|b'|$, so again using Theorem \ref{thm:lab} we obtain that $P_{b'} \leq n^{- 4 \varepsilon}$.

Now we will estimate the probability $P_{b''}$ of fulfilling $b$ using the relator $r$ two times. The only face in $b'$ which can bear $r$ is the right bottom face. We can, therefore, consider diagram $b''$ where we remove this face and label the new boundary edges with three consecutive letters of $r$ (See Figure \ref{fig:2bis}).

\begin{figure}[h!]
	\centering
		\includegraphics[scale=0.8]{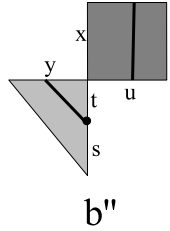}
	\caption{Diagram with five fixed edges: $x$, $y$ and $s$, $t$, $u$ are 2-tuples of consecutive letters of $r$.}
	\label{fig:2bis}
\end{figure}

Observe that diagram $b''$ satisfies: $|\partial b''| - 2 \cdot 5 < \frac{4}{3}|b''|$ so from Theorem \ref{thm:lab} (used for $l=4$) the probability of fulfilling $b''$ is less than $n^{- 4 \varepsilon}$.

For a fixed relator $r$ there are 8 possible pairs $x, y$ and also 8 possible triples $s, t, u$. Hence, we can estimate the probability $P_r$ that $r$ is a corner of a $2$-collared diagram:

$$ P < 8( P_{a'} + P_{b'} + 8 P_{b''} ) < 80 n^{- 4 \varepsilon}. $$

For each positive word $r$ of length 4 let $\widetilde{P}_r$ be the probability that the first relator equals $r$. Therefore, from the Bayes formula we can estimate the probability of fulfilling $E$ by:

$$ \sum_{r \in W_n}P_r \widetilde{P}_r \leq \sum_{r \in W_n}P_r \cdot 80 n^{- 4 \varepsilon} = 80 n^{- 4 \varepsilon},$$
since $\sum_{r \in W_n}P_r = 1$.

Therefore, choosing a $2$-cell in $\cay$ bearing the first relator and taking 2 hypergraphs passing through this cell, with overwhelming probability, gives us a pair of hypergraphs which cross only once.
\end{proof}

\begin{lem}\label{lem:leaf}
In the positive square model, at density $d > \frac{1}{4}$, all hypergraphs are leafless trees.
\end{lem}

\begin{proof}
Note, that a hypergraph can have a leaf only if there exists a generator which appears in exacly one relator. Let us consider a fixed generator $a$. The number of positive words of length 4 containing generator $a$ equals $(n^4 - (n-1)^4)$, so $(n^4 - (n-1)^4) \left((n-1)^4 \atop \lfloor  n^{4d} \rfloor - 1\right)$ is the number of presentations where exactly one relator contains $a$. Hence the probability that $a$ appears in exactly one relator equals:

$$p = \frac{(n^4 - (n-1)^4) \left((n-1)^4 \atop \lfloor n^{ 4d} \rfloor - 1\right)}{ \left(n^4 \atop \lfloor n^{ 4d} \rfloor \right)}  .$$

Since $d<1$ we have $\left(n^4 \atop \lfloor n^{ 4d} \rfloor \right) \geq \left(n^4 \atop \lfloor n^{ 4d} \rfloor - 1 \right)$, so we estimate:

$$p \leq  (n^4 - (n-1)^4) \frac{\left((n-1)^4 \atop \lfloor n^{ 4d} \rfloor - 1\right)}{\left(n^4 \atop \lfloor n^{ 4d} \rfloor - 1 \right)} $$

We continue estimation:

$$ \frac{\left((n-1)^4 \atop \lfloor n^{ 4d} \rfloor - 1\right)}{\left(n^4 \atop \lfloor n^{ 4d} \rfloor - 1 \right)} = \frac{(n-1)^4 ((n-1)^4 - 1) \dots ((n-1)^4 - \lfloor n^{4d} \rfloor) }{n^4 (n^4 - 1) \dots (n^4 - \lfloor n^{4d} \rfloor) } < \left(\frac{n-1}{n} \right)^{4 (\lfloor n^{4d} \rfloor - 1)}$$

The probability that there exists a generator that is contained in exactly one relator is bounded by $n p$. Let us denote $z_n =  n(n^4 - (n-1)^4 + 1) \left(\frac{n-1}{n} \right)^{4(\lfloor n^{4d} \rfloor - 1)}$. Note that $p n < z_n$, so it suffices to show that $\displaystyle \lim_{n \rightarrow \infty} \ln z_n = -\infty$. Note that: $z_n < 2n^5 \left(\frac{n-1}{n} \right)^{4 ( \lfloor n^{4d} \rfloor - 1 )}$. From the fact that $|\ln(1-x)| > x$ for $x \in (0,1)$, we can estimate:

\begin{equation}\label{eq:odw}
\ln z_n < \ln 2 + 5\ln n + 4\left( \lfloor n^{4d} \rfloor - 1 \right) \ln \left( \frac{n-1}{n} \right) < \ln 2 + 5\ln n - 4\left( \lfloor n^{4d} \rfloor - 1 \right)  \frac{1}{n}
\end{equation}

Since $d > \frac{1}{4}$ the right hand side of (\ref{eq:odw}) converges to $-\infty$ when $n \rightarrow \infty$.
\end{proof}

\subsection{For densities $\frac{1}{4} < d < \frac{1}{3}$ a random group in the positive square model does not have Property (T)}

\begin{defi}
For a hypergraph $\Lambda$ in $\cay$ the \textit{orientation preserving stabilizer} $\emph{Stabilizer}^+(\Lambda)$ is the index $\leq 2$ subgroup of $\emph{Stabilizer}(\Lambda)$ that also stabilizes each of the two components of $\cay - \Lambda$.
\end{defi}

We now recall the following

\begin{lemma}[{\cite[Lemma 7.2]{ow11}}]\label{thm:numb}
Suppose that a group $G$ acts cocompactly and freely on $\cay$ and the system of hypergraphs in $\cay$ is locally finite and cocompact (meaning that the hypergrahs in $\cay / G$ are compact and there is a finite number of them). Suppose that two distinct leafless hypergraphs $\Lambda_1$ and $\Lambda_2$, which are embedded trees, cross at a single point.

Then for $i=1,2$ the group $H_i = \emph{Stabilizer}^+(\Lambda_i)$ is a subgroup of $G$ with a relative number of ends $e(G,H_i) = 2$.
\end{lemma}

\begin{theo}\label{thm:faj}
In the positive square model at density $\frac{1}{4} < d < \frac{1}{3}$ with overwhelming probability a random group $G$ has a subgroup $H$ which is free, quasiconvex and such that the relative number of ends $e(G, H)$ is at least 2. In particular with overwhelming probability $G$ does not have Property \emph{(T)}.
\end{theo}

\begin{proof}
From Lemma \ref{lem:et} we know that with overwhelming probability hypergraphs in $\cay$ are embedded trees. Then from Corollary \ref{cor:sub} we know that with overwhelming probability $\emph{Stabilizer}(\Lambda)$ of any hypergraph $\Lambda$ is a free, quasiconvex subgroup, hence so is $\emph{Stabilizer}^+(\Lambda)$ since it is a subgroup in $\emph{Stabilizer}(\Lambda)$ of index $\leq 2$.

The presentation complex $X$ of $G$ is finite since our group is finitely presented so $G$ acts cocompactly on $\cay$ and the system of hypergraphs is locally finite and cocompact. From Theorem \ref{thm:once} we know that with overwhelming probability there is a pair of hypergraphs which intersect exactly once and from Lemma \ref{lem:leaf} we have that with overwhelming probability all hypergraphs are leafless trees.

Hence, from Lemma \ref{thm:numb} we have that with overwhelming probability there is a subgroup in $G$ which is free, quasiconvex and that the relative number of ends $e(G, H)$ is at least 2.

Finally using Theorem \ref{thm:criterion} we get that $G$ does not have Property (T).
\end{proof}

\subsection{Groups without Property (T) in the square model}

We proved in Section \ref{sec:iso} that the isoperimetric inequality holds in the square model. Some of the proofs of lemmas and theorems in the previous section can be generalized to the square model.

\begin{lem}\label{lem:ets} In the square model for densities $< \frac{1}{3}$ the hypergraphs in the Cayley complex of a random group are embedded trees.
\end{lem}

\begin{proof}
The proof is identical to the proof of Lemma \ref{lem:et}.
\end{proof}

Let $\cay$ be the Cayley complex of the random group in the square model at density $d$.

\begin{lem}\label{lem:leafs}
In the square model at density $d > \frac{1}{4}$  all hypergraphs are leafless trees.
\end{lem}

\begin{proof}
The proof is analogous to the proof of \ref{lem:leaf}.
\end{proof}

\begin{theo}\label{thm:onces}
With overwhelimg probability, in the square model, there exists a pair of hypergraphs $\Lambda_1$, $\Lambda_2$ in $\cay$ such that $\Lambda_1 \cap \Lambda_2$ intersect only once.
\end{theo}

\begin{proof}
The proof is analogous to the proof of Theorem \ref{thm:once}.
\end{proof}

Combining Theorem \ref{thm:onces} with Lemma \ref{lem:leafs}, by Lemma \ref{thm:numb} and Theorem \ref{thm:criterion} we obtain the following:

\begin{theo}\label{thm:nT}
In the square model at density $d < \frac{1}{3}$ with overwhelming probability a random group does not have property \emph{(T)}.
\end{theo}


\begin{bibdiv}
\begin{biblist}

   \bib{bdl}{book}{
   author={Bekka, Bachir},
   author={de la Harpe, Pierre},
   author={Valette, Alain},
   title={Kazhdan's property (T)},
   series={New Mathematical Monographs},
   volume={11},
   publisher={Cambridge University Press},
   place={Cambridge},
   date={2008},
   pages={xiv+472}}

   \bib{bela}{book}{
   author={Bollob{\'a}s, B{\'e}la},
   title={Random graphs},
   series={Cambridge Studies in Advanced Mathematics},
   volume={73},
   edition={2},
   publisher={Cambridge University Press},
   place={Cambridge},
   date={2001},
   pages={xviii+498},
   isbn={0-521-80920-7},
   isbn={0-521-79722-5},
   review={\MR{1864966 (2002j:05132)}},
   doi={10.1017/CBO9780511814068},
}

    \bib{er59}{article}{
   author={Erd{\"{o}}s, P.},
   author={R{\'e}nyi, A.},
   title={On random graphs. I},
   journal={Publ. Math. Debrecen},
   volume={6},
   date={1959},
   pages={290--297}}

   \bib{er60}{article}{
   author={Erd{\"{o}}s, P.},
   author={R{\'e}nyi, A.},
   title={On the evolution of random graphs},
   language={English, with Russian summary},
   journal={Magyar Tud. Akad. Mat. Kutat\'o Int. K\"ozl.},
   volume={5},
   date={1960},
   pages={17--61}}


   \bib{gro93}{article}{
   author={Gromov, M.},
   title={Asymptotic invariants of infinite groups},
   conference={
      title={Geometric group theory, Vol.\ 2},
      address={Sussex},
      date={1991},
   },
   book={
      series={London Math. Soc. Lecture Note Ser.},
      volume={182},
      publisher={Cambridge Univ. Press},
      place={Cambridge},
   },
   date={1993},
   pages={1--295}}

   
   \bib{kot11}{article}{
   author={Kotowski, Marcin},
   author={Kotowski, Micha{\l}},
   title={Random groups and property $(T)$: \.Zuk's theorem revisited},
   journal={J. Lond. Math. Soc. (2)},
   volume={88},
   date={2013},
   number={2},
   pages={396--416},
}

   \bib{ls}{book}{
   author={Lyndon, Roger C.},
   author={Schupp, Paul E.},
   title={Combinatorial group theory},
   note={Ergebnisse der Mathematik und ihrer Grenzgebiete, Band 89},
   publisher={Springer-Verlag},
   place={Berlin},
   date={1977},
   pages={xiv+339}}


   \bib{nr98}{article}{
   author={Niblo, Graham A.},
   author={Roller, Martin A.},
   title={Groups acting on cubes and Kazhdan's property (T)},
   journal={Proc. Amer. Math. Soc.},
   volume={126},
   date={1998},
   number={3},
   pages={693--699}}


   \bib{oll05}{book}{
   author={Ollivier, Yann},
   title={A January 2005 invitation to random groups},
   series={Ensaios Matem\'aticos [Mathematical Surveys]},
   volume={10},
   publisher={Sociedade Brasileira de Matem\'atica},
   place={Rio de Janeiro},
   date={2005},
   pages={ii+100}}

   \bib{oll-f}{article}{
   author={Ollivier, Yann},
   title={Some small cancellation properties of random groups},
   journal={Internat. J. Algebra Comput.},
   volume={17},
   date={2007},
   number={1},
   pages={37--51}}

   \bib{ow11}{article}{
   author={Ollivier, Yann},
   author={Wise, Daniel T.},
   title={Cubulating random groups at density less than $1/6$},
   journal={Trans. Amer. Math. Soc.},
   volume={363},
   date={2011},
   number={9},
   pages={4701--4733}}

   \bib{sw77}{article}{
   author={Scott, Peter},
   author={Wall, Terry},
   title={Topological methods in group theory},
   conference={
      title={Homological group theory},
      address={Proc. Sympos., Durham},
      date={1977},
   },
   book={
      series={London Math. Soc. Lecture Note Ser.},
      volume={36},
      publisher={Cambridge Univ. Press},
      place={Cambridge},
   },
   date={1979},
   pages={137--203}}

   \bib{serfling}{article}{
   author={Serfling, R. J.},
   title={Probability inequalities for the sum in sampling without
   replacement},
   journal={Ann. Statist.},
   volume={2},
   date={1974},
   pages={39--48},
   issn={0090-5364},
   review={\MR{0420967 (54 \#8976)}},
}

     \bib{serr}{book}{
   author={Serre, Jean-Pierre},
   title={Trees},
   note={Translated from the French by John Stillwell},
   publisher={Springer-Verlag},
   place={Berlin},
   date={1980},
   pages={ix+142}}


   \bib{zuk96}{article}{
   author={{\.Z}uk, Andrzej},
   title={La propri\'et\'e (T) de Kazhdan pour les groupes agissant sur les
   poly\`edres},
   language={French, with English and French summaries},
   journal={C. R. Acad. Sci. Paris S\'er. I Math.},
   volume={323},
   date={1996},
   number={5},
   pages={453--458}}



\bib{zuk03}{article}{
   author={{\.Z}uk, Andrzej},
   title={Property (T) and Kazhdan constants for discrete groups},
   journal={Geom. Funct. Anal.},
   volume={13},
   date={2003},
   number={3},
   pages={643--670}}





\end{biblist}
\end{bibdiv}
\end{document}